\documentclass[]{amsart}

\usepackage{graphicx}
\usepackage{pinlabel}

\newtheorem{theorem}{Theorem}
\newtheorem{lemma}[theorem]{Lemma}
\newtheorem{corollary}[theorem]{Corollary}

\begin{document}

\title{Cyclic polygons in classical geometry}

\author{Ren Guo}

\address{School of Mathematics, University of Minnesota, Minneapolis, MN, 55455, USA}

\email{guoxx170@math.umn.edu}

\author{Nilg\"un S\"onmez}

\address{Afyon Kocatepe University, Faculty of Science and Arts, Department of Mathematics, ANS Campus, 03200 - Afyonkarahisar, Turkey}

\email{nceylan@aku.edu.tr}



\begin{abstract} Formulas about the side lengths, diagonal lengths or radius of the circumcircle of a cyclic polygon in Euclidean geometry, hyperbolic geometry or spherical geometry can be unified. 
\end{abstract}

\maketitle

\section{Introduction}

In Euclidean geometry, hyperbolic geometry or spherical geometry, a cyclic polygon is a polygon whose vertexes are on a same circle. In Euclidean geometry, the side lengths and diagonal lengths of a cyclic polygon satisfy some polynomials. Ptolemy's theorem about a cyclic quadrilateral and Fuhrmann's theorem about a cyclic hexagon are examples. The two theorems also hold in hyperbolic geometry, for example, see \cite{S}. It is also observed in \cite{S} that the formulas for hyperbolic geometry are easily obtained by replacing an edge length $l/2$ in Euclidean geometry by $\sinh l/2$. In the paper we will show this is a general principle to translate a result in Euclidean geometry to a result in hyperbolic geometry. Also we get formulas in spherical geometry by replacing an edge length $l/2$ in Euclidean geometry by $\sin l/2$. 

In Euclidean geometry, the radius of the circumcircle of a cyclic polygon can be calculated from the side lengths. For more information about the radius of the circumcircle of a cyclic polygon, see, for example, \cite{MRR,FP,P,V}. As a corollary of our main result, the formulas of radius of the circumcircle of a cyclic polygon in Euclidean geometry, hyperbolic geometry or spherical geometry can be unified. 

In this paper, we do not give any new theorems about a cyclic polygon in Euclidean geometry. But we show that once there is a formula about sides lengths, diagonal lengths and radius of circumcircle of a cyclic polygon in Euclidean geometry, there is an essentially same formula in hyperbolic geometry or spherical geometry. In other words, we provide a machinery to generate theorems in hyperbolic geometry or spherical geometry. 

For recent development of the study of cyclic polygons in Euclidean geometry, see, for example, \cite{P1,P2}.

In section 2, 3 and 4, the general principle are illustrated by examples about triangles, cyclic quadrilaterals and cyclic hexagons. In section 5, the main result, Theorem \ref{main}, is stated formally. Section 6 establishes a lemma and section 7 proves Theorem \ref{main}.

\section{Triangle}

For a triangle on the Euclidean plane, assume that it has the side lengths $a,b,c$ and the radius of circumcircle $r$. Then 
$$r^2=\frac{(abc)^2}{(a+b+c)(b+c-a)(a+c-b)(a+b-c)}.$$ 

For a triangle on the hyperbolic plane, assume that it has the side lengths $a,b,c$ and the radius of circumcircle $r$. Then 
$$\frac14\sinh^2 r=\frac{(\sinh\frac a2 \sinh\frac b2 \sinh\frac c2)^2}
{A_h(A_h-2\sinh\frac a2)(A_h-2\sinh\frac b2)(A_h-2\sinh\frac c2)},$$ 
where $A_h=\sinh\frac a2+\sinh\frac b2+\sinh\frac c2.$

For a triangle on the unit sphere, assume that it has the side lengths $a,b,c$ and the radius of circumcircle $r$. Then 
$$\frac14\sin^2 r=\frac{(\sin\frac a2 \sin\frac b2 \sin\frac c2)^2}
{A_s(A_s-2\sin\frac a2)(A_s-2\sin\frac b2)(A_s-2\sin\frac c2)},$$ 
where $A_s=\sin\frac a2+\sin\frac b2+\sin\frac c2.$
The three formulas are essentially the same. They can be unified by introduce the following function:

$$
s(x)=\left\{
\begin{array}{lll}
\frac x2             &\ \ \mbox{in Euclidean geometry} \\
\sinh \frac x2       &\ \ \mbox{in hyperbolic geometry}\\
\sin  \frac x2       &\ \ \mbox{in spherical geometry}
\end{array}
\right.
$$

The unified formula is
$$\frac14 s(2r)^2=\frac{(s(a)s(b)s(c))^2}{A_3(A_3-2s(a))(A_3-2s(b))(A_3-2s(c))},$$ 
where $A_3=s(a)+s(b)+s(c)$.

\section{Cyclic quadrilateral}

\begin{figure}[ht!]
\labellist\small\hair 2pt
\pinlabel $a$ at 133 136
\pinlabel $b$ at 19 136
\pinlabel $c$ at 54 28
\pinlabel $d$ at 160 28
\pinlabel $e$ at 43 78
\pinlabel $f$ at 86 128

\endlabellist
\centering
\includegraphics[scale=0.45]{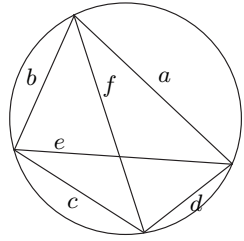}
\caption{Ptolemy's theorem}
\label{fig:quad}
\end{figure}

We can observe the similar phenomenon in the case of a quadrilateral inscribed in a circle which is called a cyclic quadrilateral. That means the formula about the side lengths, diagonal lengths and radius of circumcircle of a cyclic quadrilateral can be unified in Euclidean geometry, hyperbolic geometry and spherical geometry. Assume the side lengths are $a,b,c,d$, the diagonal lengths are $e,f$ as labeled in Figure \ref{fig:quad} and the radius of the circumcicle is $r$. Then the Ptolemy's theorem says
$$s(e)s(f)=s(a)s(c)+s(b)s(d).$$

For more information about generalization of Ptolemy's theorem in hyperbolic and spherical geometry, see \cite{V1,V2}.

The formulas representing diagonal lengths in terms of side lengths are 
$$s(e)^2=(s(a)s(c)+s(b)s(d))\frac{s(a)s(d)+s(b)s(c)}{s(a)s(b)+s(c)s(d)},$$
$$s(f)^2=(s(a)s(c)+s(b)s(d))\frac{s(a)s(b)+s(c)s(d)}{s(a)s(d)+s(b)s(c)}.$$ 
Ptolemy's theorem is a corollary of the two formulas.
               
The formula involving the radius is                      
$$\frac14 s(2r)^2=\frac{(s(a)s(b)+s(c)s(d))(s(a)s(c)+s(b)s(d))(s(a)s(d)+s(b)s(c))}
{(A_4-2s(a))(A_4-2s(b))(A_4-2s(c))(A_4-2s(d))}$$
where $A_4=s(a)+s(b)+s(c)+s(d).$

\section{Cyclic hexagon}

\begin{figure}[ht!]
\labellist\small\hair 2pt
\pinlabel $a$  at 162 187
\pinlabel $a'$ at 23 19
\pinlabel $b$  at 73 188
\pinlabel $b'$ at 141 3
\pinlabel $c$  at 9 120
\pinlabel $c'$ at 181 89
\pinlabel $e$  at 86 95
\pinlabel $f$  at 94 130
\pinlabel $g$  at 115 106

\endlabellist
\centering
\includegraphics[scale=0.45]{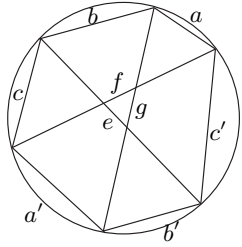}
\caption{Fuhrmann's theorem}
\label{fig:hex}
\end{figure}

One more example is Fuhrmann's theorem about a cyclic hexagon in Euclidean geometry. This theorem also holds in hyperbolic and spherical geometry. And the formula can be unified in the three cases. Assume that a convex cyclic hexagon in Euclidean geometry, hyperbolic geometry and spherical geometry have the side lengths $a, a', b, b', c,$ and $c'$, and diagonal lengths $e, f$, and $g$ as labeled in Figure \ref{fig:hex}, then
$$s(e)s(f)s(g)=s(a)s(a')s(e)+s(b)s(b')s(f)+s(c)s(c')s(g)+s(a)s(b)s(c)+s(a')s(b')s(c').$$

The generalization of Fuhrmann's theorem into hyperbolic geometry is treated in Wilson Stothers' web page about hyperbolic geometry \cite{S}.

\section{Cyclic polygon}

We can expect that the similar phenomenon holds for general polygon inscribed in a circle which is called the cyclic polygon. In this paper we show that the same relationships about the side lengths, diagonal lengths and the radius of the circumcicle hold in Euclidean geometry, hyperbolic geometry and spherical geometry. To make the statement formal, we introduce the following notations.  

Fix an integer $n\geq 3.$

The set $\mathcal{E}_n$ of polynomials is defined as follows. A polynomial $f$ of $\frac{n(n+1)}{2}$ variables is in the set $\mathcal{E}_n$ if $f(|P_iP_j|_e,r_e)=0$ for any $n$ points $P_1,P_2,...,P_n$ on a circle of radius $r_e$ on the Euclidean plane, where $|P_iP_j|_e$ denotes the Euclidean distance between the two points $P_i$ and $P_j$.   

The set $\mathcal{H}_n$ of polynomials is defined as follows. A polynomial $f$ of $\frac{n(n+1)}{2}$ variables is in the set $\mathcal{H}_n$ if $f(\sinh\frac{|P_iP_j|_h}2,\frac12\sinh r_h)=0$ for any $n$ points $P_1,P_2,...,P_n$ on a circle of radius $r_h$ on the hyperbolic plane, where $|P_iP_j|_h$ denotes the hyperbolic distance between the two points $P_i$ and $P_j$. 

The set $\mathcal{S}_n$ of polynomials is defined as follows. A polynomial $f$ of $\frac{n(n+1)}{2}$ variables is in the set $\mathcal{S}_n$ if $f(\sin\frac{|P_iP_j|_s}2,\frac12\sin r_s)=0$ for any $n$ points $P_1,P_2,...,P_n$ on a circle of radius $r_s$ on the unit sphere, where $|P_iP_j|_s$ denotes the spherical distance between the two points $P_i$ and $P_j$. 

\begin{theorem}\label{main} $\mathcal{E}_n=\mathcal{H}_n=\mathcal{S}_n$.
\end{theorem}

R. J. Gregorac \cite{G} generalized Ptolemy's theorem to a convex cyclic polygon on the Euclidean plane and generalized Fuhrmann's theorem to a convex cyclic $2n$-gon on the Euclidean plane. By Theorem \ref{main}, Gregorac's two results can be easily generalized into hyperbolic and spherical geometry and the formulas can be unified. For example, we consider the later one.

\begin{corollary} Let $n>3.$ Let $\{V_0,V_1,...,V_{2n-1}\}$ be points on a circle in Euclidean geometry, hyperbolic geometry or spherical geometry. Let $l_{ij}$ denote the length of the geodesic segment $V_iV_j$ for $i\neq j$. Then $$\det(a_{ij})=0$$ where 
$$a_{ji}=(-1)^{\delta_{ij}}(s(l_{2i-2,2j-1})s(l_{2j-1,2i}))^{-1}$$ and $\delta_{ij}$ is the Kronecker delta. 
\end{corollary}

A cyclic polygon in Euclidean, hyperbolic or spherical geometry is uniquely determined by its side lengths. Therefore any diagonal length or radius of the circumcircle is a function of side lengths. In the case of Euclidean geometry, it is not difficult to see that these functions are algebraic functions. As a corollary of Theorem \ref{main}, these functions are unified in Euclidean, hyperbolic or spherical geometry.

\begin{corollary} For a Euclidean, hyperbolic or spherical cyclic polygon with vertexes $P_1,P_2,...,P_n$ and the side lengths $|P_iP_{i+1}|=l_{i,i+1}$ where $n+1=1$, the length $l_{ij}$ of the diagonal $P_iP_j$ $(|j-i|\geq 2)$ is
$$s(l_{ij})=F_{ij}(s(l_{12}),s(l_{23}),...,s(l_{n1}))$$ for some algebraic function $F_{ij}$ which is independent of the three kinds of geometry. The radius $r$ of the circumcircle circle satisfies $$s(2r)=G(s(l_{12}),s(l_{23}),...,s(l_{n1}))$$ for some algebraic function $G$ which is independent of the three kinds of geometry.
\end{corollary}

\section{Homogeneity}

\begin{figure}[ht!]
\labellist\small\hair 2pt
\pinlabel $P_1$  at 132 84
\pinlabel $P_1'$ at 156 74
\pinlabel $P_2$  at 101 172
\pinlabel $P_2'$ at 111 198
\pinlabel $P_3$  at 30 153
\pinlabel $P_3'$ at 9 172
\pinlabel $P_4$  at 33 76
\pinlabel $P_4'$ at 12 56
\pinlabel $P_1$  at 344 86
\pinlabel $P_1'$ at 368 73
\pinlabel $P_2$  at 314 174
\pinlabel $P_2'$ at 324 202
\pinlabel $P_3$  at 243 153
\pinlabel $P_3'$ at 222 174
\pinlabel $P_4$  at 246 76
\pinlabel $P_4'$ at 228 55
\pinlabel $P_1$  at 542 106
\pinlabel $P_1'$ at 525 76
\pinlabel $P_2$  at 566 125
\pinlabel $P_2'$ at 567 88
\pinlabel $P_3$  at 476 128
\pinlabel $P_3'$ at 470 88
\pinlabel $P_4$  at 455 108
\pinlabel $P_4'$ at 464 56
\pinlabel $(a)$ at 81 10
\pinlabel $(b)$ at 297 10
\pinlabel $(c)$ at 507 10
\pinlabel $O$ at 74 122
\pinlabel $O$ at 285 122
\pinlabel $S$ at 503 32
\endlabellist
\centering
\includegraphics[scale=0.62]{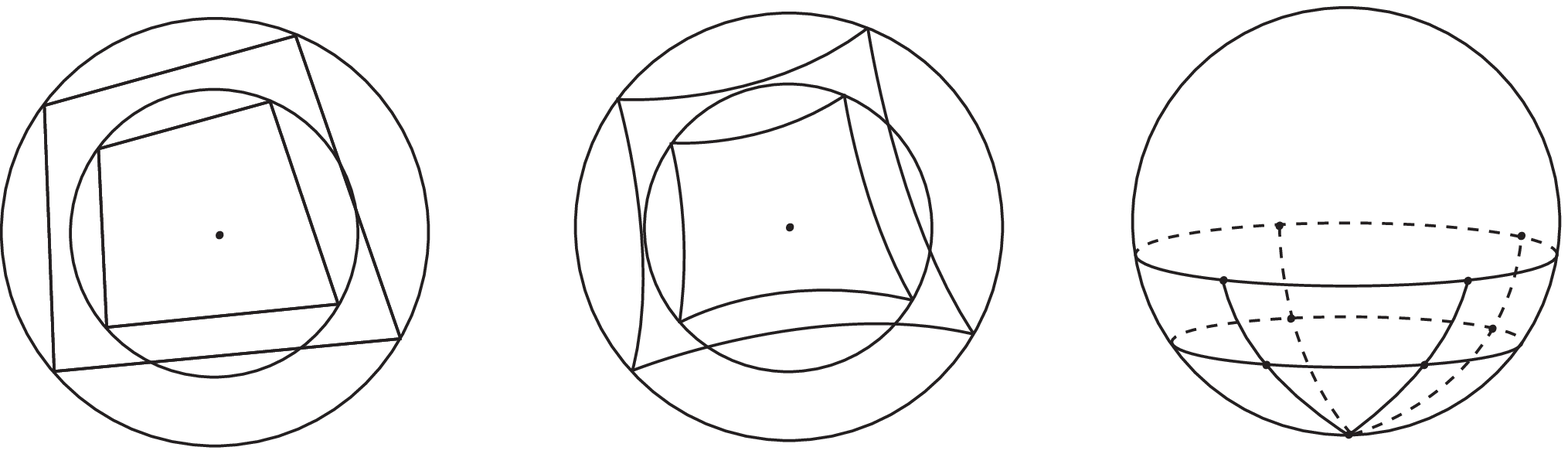}
\caption{}
\label{fig:homo}
\end{figure}

To prove Theorem \ref{main}, we need the following lemma.

\begin{lemma}\label{homo} If $f\in \mathcal{E}_n \cup \mathcal{H}_n \cup \mathcal{S}_n$, then $f$ is homogeneous.
\end{lemma}

\begin{proof} Assume $f\in \mathcal{E}_n$. Then $f(|P_iP_j|_e,r_e)=0$ for any $n$ points $P_1,P_2,...,P_n$ on a circle of any radius $r_e$ on the Euclidean plane. We can assume the center of the circle is the origin. For any $x>0,$ by multiplying the coordinates of $P_1,P_2,...,P_n$ by $x$, we obtain new points $P'_1,P'_2,...,P'_n$ with is on the circle of radius $xr_e$ centered at the origin, see Figure \ref{fig:homo}(a). The distance $|P'_iP'_j|_e=x|P_iP_j|_e$ and the radius $xr_e$ also satisfy the polynomial $f$, i.e., $f(x|P_iP_j|_e,xr_e)=0$. Therefore $f$ is homogeneous. 

Assume $f\in \mathcal{H}_n$. Then $f(\sinh\frac{|P_iP_j|_h}2,\frac12\sinh r_h)=0$ for any points $P_1,P_2,...,P_n$ on a circle $C$ of radius $r_h$ on the hyperbolic plane. We use the Poincar\'e disk model of the hyperbolic plane: $\{z\in \mathbb{C}: |z|<1\}$. Assume the center of the circle is the origin $O$. For any $x>0$, there is another circle $C'$ centered at $O$ with radius $r_h'$ satisfying $\sinh r_h'= x\sinh r_h.$ For each $i$, let $P_i'$ be the intersection point of $C'$ and the geodesic ray starting from $O$ and passing though $P_i$ as in Figure \ref{fig:homo}(b). By the law of sine of a hyperbolic triangle, we have 
$$\frac{\sinh\frac{|P_iP_j|_h}2}{\sinh r_h}=\sin\frac{\angle P_iOP_j}{2}
=\frac{\sinh\frac{|P_i'P_j'|_h}2}{\sinh r_h'}=\frac{\sinh\frac{|P_i'P_j'|_h}2}{x\sinh r_h}$$
where $\angle P_iOP_j \in (0,\pi].$ Therefor $\sinh\frac{|P_i'P_j'|_h}2=x\sinh\frac{|P_iP_j|_h}2$.

Since $f\in \mathcal{H}_n,$ the equation $f(\sinh\frac{|P_i'P_j'|_h}2,\frac12\sinh r_h')=0$ holds. Thus $$f(x\sinh\frac{|P_iP_j|_h}2,x\frac12\sinh r_h)=0$$ holds for any $x>0$. Therefore $f$ is homogeneous. 

Assume $f\in \mathcal{S}_n$. Then $f(\sin\frac{|P_iP_j|_s}2,\frac12\sin r_s)=0$ for any points $P_1,P_2,...,P_n$ on a circle $C$ of radius $r_s$ on the unit sphere. Assume the center of the circle is the south pole $S$. For any $x\in (0,1)$, there is another circle $C'$ centered at south pole with radius $r_s'$ satisfying $\sin r_s'= x\sin r_s.$ For each $i$, let $P_i'$ be the intersection point of $C'$ and the geodesic ray starting from $O$ and passing through $P_i$ as in Figure \ref{fig:homo}(c). By the law of sine of a spherical triangle, we have 
$$\frac{\sin\frac{|P_iP_j|_s}2}{\sin r_s}=\sin\frac{\angle P_iOP_j}{2}
=\frac{\sin\frac{|P_i'P_j'|_s}2}{\sin r_s'}=\frac{\sin\frac{|P_i'P_j'|_s}2}{x\sin r_s}$$
where $\angle P_iOP_j \in (0,\pi].$ Therefor $\sin\frac{|P_i'P_j'|_s}2=x\sin\frac{|P_iP_j|_s}2$.

Since $f\in \mathcal{S}_n,$ the equation $f(\sin\frac{|P_i'P_j'|_s}2,\frac12\sin r_s')=0$ holds. Thus $$f(x\sin\frac{|P_iP_j|_s}2,x\frac12\sin r_s)=0$$ holds for any $x\in (0,1)$. Therefore $f$ is homogeneous. 

\end{proof}

\section{Proof of Theorem \ref{main}}

\subsection{$\mathcal{E}_n \subseteq \mathcal{H}_n$} 

Let $f$ be a polynomial in $\mathcal{E}_n.$ 

Consider $n$ points $P_1,P_2,...,P_n$ on a circle of radius $r_h$ on the hyperbolic plane. We use the Poincar\'e disk model of the hyperbolic plane. Assume that the center of the circle is the origin $O$. In the hyperbolic triangle $\triangle_h P_iOP_j$ which can be degenerated into a line segment, by the law of sine, we have 
$$\sin\frac{\angle P_iOP_j}{2}=\frac{\sinh\frac{|P_iP_j|_h}2}{\sinh r_h}$$
where $\angle P_iOP_j \in (0,\pi].$

On the other hand, this circle is also a Euclidean circle. We consider the Euclidean triangle $\triangle_e P_iOP_j$, by the law of sine, we have 
$$\sin\frac{\angle P_iOP_j}{2}=\frac{|P_iP_j|_e}{2r_e}.$$

Hence $$\sinh\frac{|P_iP_j|_h}2=\frac{\sinh r_h}{2r_e}\cdot |P_iP_j|_e.$$
Obviously,
$$\frac12 \sinh r_h=\frac{\sinh r_h}{2r_e}\cdot r_e.$$

Since $f\in \mathcal{E}_n,$ $f(|P_iP_j|_e,r_e)=0$. By Lemma \ref{homo}, 
$$f(\sinh\frac{|P_iP_j|_h}2,\frac12\sinh r_h)=f(x|P_iP_j|_e,xr_e)=0,$$ where $x=\frac{\sinh r_h}{2r_e}$. Hence $f\in \mathcal{H}_n$.

\subsection{$\mathcal{H}_n \supseteq \mathcal{E}_n$} 

Let $f$ be a polynomial in $\mathcal{H}_n.$ 

For any points $P_1,P_2,...,P_n$ on a circle of radius $r_e$ centered at the origin on the Euclidean plane. For a number $y\in (0,\frac 1{r_e}),$ consider the circle $C'$ centered at the origin and with radius $yr_e.$ 
For any $i$, let $P'_i$ be the the intersection point of $C'$ and the ray starting from $O$ and and passing through $P_i$. For any $i,j$, we consider the Euclidean triangle $\triangle_e P_iOP_j$ which can be degenerated. On the other hand,
Since all the points $P'_1,P'_2,...,P'_n$ are in the unit disk which is considered as the hyperbolic plane, for any $i,j$, there is also a hyperbolic triangle $\triangle_h P_iOP_j$.
By the law of sine, we have 
$$\frac{\sinh\frac{|P'_iP'_j|_h}2}{\sinh |OP'_i|_h}=\sin\frac{\angle P_iOP_j}{2}=\frac{|P'_iP'_j|_e}{2yr_e}=\frac{|P_iP_j|_e}{2r_e}.$$ 

By Lemma \ref{homo}, $f(x\sinh\frac{|P'_iP'_j|_h}2,x\frac12\sinh |OP'_i|_h)=0$ for any $x>0$. By taking $x=\frac{2r_e}{\sinh |OP'_i|_h}$, 
we have $f(|P_iP_j|_e, r_e)=0.$ Thus $f\in \mathcal{E}_n.$

To sum up, we have proved that $\mathcal{E}_n=\mathcal{H}_n$.

\subsection{$\mathcal{E}_n \subseteq \mathcal{S}_n$}

Let $f$ be a polynomial in $\mathcal{E}_n.$ 

Consider $n$ points $P_1,P_2,...,P_n$ on a circle of radius $r_s$ on the unit sphere $\{(x,y,z):x^2+y^2+z^2=1\}$. Assume that the center of the circle is the south pole $S$. In the spherical triangle $\triangle_s P_iSP_j$ which can be degenerated, by the law of sine, we have 
$$\sin\frac{\angle P_iSP_j}{2}=\frac{\sin\frac{|P_iP_j|_s}2}{\sin r_s}.$$

For any $i$, let $P'_i$ be the image of $P_i$ under the stereographic projection from the north pole, i.e., the intersection point of the plane $z=0$ and the line passing through the north pole and $P_i$. Then $P'_1,P'_2,...,P'_n$ are points on a circle of radius $r_e$ centered at the origin $O$ on the Euclidean plane $z=0$. In the Euclidean triangle $\triangle_e P_iOP_j$, by the law of sine, we have 
$$\sin\frac{\angle P'_iOP'_j}{2}=\frac{|P_iP_j|_e}{2r_e}.$$

Since $\angle P_iSP_j=\angle P'_iOP'_j,$ we have 
$$\sin\frac{|P_iP_j|_s}2=\frac{\sin r_s}{2r_e}\cdot |P'_iP'_j|_e.$$
Obviously,
$$\frac12 \sin r_s=\frac{\sin r_s}{2r_e}\cdot r_e.$$

Since $f\in \mathcal{E}_n,$ $f(|P'_iP'_j|_e,r_e)=0$. By Lemma \ref{homo}, 
$$f(\sin\frac{|P_iP_j|_s}2,\frac12\sin r_s)=f(x|P'_iP'_j|_e,xr_e)=0,$$ where $x=\frac{\sin r_s}{2r_e}$. Hence $f\in \mathcal{S}_n$.

\subsection{$\mathcal{S}_n \subseteq \mathcal{E}_n$}

Let $g$ be a polynomial in $\mathcal{S}_n.$

For any points $B_1,B_2,...,B_n$ on a circle of radius $r_e$ centered at the origin $O$ on the Euclidean plane. 
Consider the stereographic projection from the north pole of the unit sphere $\{(x,y,z):x^2+y^2+z^2=1\}$. The image of 
$B_1,B_2,...,B_n$ under the stereographic projection are points $B'_1,B'_2,...,B'_n$ on a circle of radius $r_s$ centered at the south pole $S$ on the unit sphere. By the law of sine, we have 
$$\sin\frac{\angle B_iOB_j}{2}=\frac{|B_iB_j|_e}{2r_e},$$ 
$$\sin\frac{\angle B'_iSB'_j}{2}=\frac{\sin\frac{|B'_iB'_j|_s}2}{\sin r_s}.$$

Since $\angle B_iOB_j=\angle B'_iSB'_j,$ we have 
$$|B_iB_j|_e=\frac{2r_e}{\sin r_s}\cdot \sin\frac{|B'_iB'_j|_s}2.$$
Obviously, 
$$r_e=\frac{2r_e}{\sin r_s} \cdot \frac12 \sin r_s.$$

Since $g\in \mathcal{S}_n,$ $g(\sin\frac{|B'_iB'_j|_s}2, \frac12 \sin r_s)=0.$ By Lemma \ref{homo}, $$g(x\sin\frac{|B'_iB'_j|_s}2, x\frac12 \sin r_s)=g(|B_iB_j|_e, r_e)=0,$$ where $x=\frac{2r_e}{\sin r_s}.$ Hence $g\in  \mathcal{E}_n.$ 

To sum up, we have proved that $\mathcal{E}_n=\mathcal{S}_n$.

\end{document}